\documentclass[twoside,11pt,reqno]{amsart}
\usepackage[a4paper,margin=1.2in]{geometry}
\usepackage[dvipsnames]{xcolor}
\usepackage{amsthm,amsmath,amsfonts,graphicx,geometry,lipsum,amsxtra}
\usepackage{hyperref}
\usepackage{mathrsfs}
\usepackage{verbatim}
\usepackage{appendix}
\usepackage{tikz}
\usepackage{epigraph}
\usepackage{tcolorbox}

\numberwithin{equation}{section}
\usepackage{esint}
\usepackage{mathtools}

\usepackage[dvipsnames]{xcolor}
\usepackage{etoolbox}
\hypersetup{
    colorlinks=true,
    linkcolor=blue,
    citecolor=blue,
    urlcolor=red,
    pdftitle={[BMSS25]Resonance near a degenerate embedded eigenvalue},
    }
\usepackage{comment}
\usepackage{amsmath}
\usepackage{tikz}
\usepackage{epigraph}
\usepackage{lipsum}
\allowdisplaybreaks

\definecolor{titlepagecolor}{cmyk}{1,.60,0,.40}

\patchcmd{\subsection}{\normalfont}{\normalfont\color{blue}}{}{}
\DeclareFixedFont{\titlefont}{T1}{ppl}{b}{it}{0.5in}

\def\th@plain{
  \thm@notefont{}
  \itshape 
}
\def\th@definition{
  \thm@notefont{}
  \normalfont 
}
\makeatother

\usepackage{srcltx} 

\usepackage{amscd}
\usepackage{amssymb}

\usepackage{amsmath}
\usepackage{pb-diagram}
\usepackage{graphicx}

\usepackage{hyperref}
\usepackage{array}
\usepackage[all]{xy}
\xyoption{matrix}
\xyoption{arrow}
\usepackage{pdfsync}
\usepackage{physics}
\usepackage{enumitem}
 
\usepackage{thmtools}
\usepackage{amssymb}

\DeclareUnicodeCharacter{2212}{-,+}
 \numberwithin{equation}{section}
\newtheorem{theorem}{Theorem}[section]
\newtheorem{proposition}[theorem]{Proposition}
\newtheorem{corollary}[theorem]{Corollary}
\newtheorem{lemma}[theorem]{Lemma}
\newtheorem{remark}[theorem]{Remark}
\newtheorem{definition}[theorem]{Definition}

\newcommand{\bdfn}{\begin{definition}}
\newcommand{\bthm}{\begin{theorem}}
\newcommand{\blem }{\begin{lemma}}
\newcommand{\bcla }{\begin{cla}aim}
\newcommand{\ecla }{\end{cla}im}
\newcommand{\bpro}{\begin{proposition}}
\newcommand{\bcor }{\begin{corollary}}

\newcommand{\brmrk}{\begin{remark}}
\newcommand{\bassu}{\begin{assumption}}
\newcommand{\eassu}{\end{assumption}}
\newcommand{\edfn}{\end{definition}}
\newcommand{\ethm}{\end{theorem}}
\newcommand{\elem }{\end{lemma}}
\newcommand{\epro}{\end{proposition}}
\newcommand{\ecor}{\end{corollary}}
\newcommand{\ermrk}{\end{remark}}

\newcommand{\RR}{\mathbb R}

\renewcommand{\iota}{{i\mkern1mu}}
\renewcommand{\Re}{\operatorname{Re}}
\renewcommand{\Im}{\operatorname{Im}}

\calclayout
\begin{document}
\title[Resonance Near a Simple Embedded Eigenvalue]{Multi-parameter Perturbations of the Laplacian and Resonance Near a Simple Embedded Eigenvalue}
\subjclass[2020]{47A10; 47A55; 81Q15. }
\keywords{resonance; embedded eigenvalue; time delay. }

\author[Bansal]{Hemant Bansal}
\address{Hemant Bansal\\ Department of Mathematical Sciences, Indian Institute of Science Education and Research Mohali\\Sector 81, SAS Nagar, Punjab 140306, India}
\email{ph20030@iisermohali.ac.in}

\author[Maharana]{Alok Maharana}
\address{Alok Maharana\\ Department of Mathematical Sciences, Indian Institute of Science Education and Research Mohali\\Sector 81, SAS Nagar, Punjab 140306, India}
\email{maharana@iisermohali.ac.in}

\author[Sahu]{Lingaraj Sahu}
\address{Lingaraj Sahu\\ Department of Mathematical Sciences, Indian Institute of Science Education and Research Mohali\\Sector 81, SAS Nagar, Punjab 140306, India}
\email{lingaraj@iisermohali.ac.in}

\vspace*{-1cm}

 \maketitle
 
\begin{abstract}
This paper continues the study of resonance phenomena initiated in \cite{LS1} for rank-one perturbations. We consider finite-rank multi-parameter perturbations $H_\alpha$ of the Laplacian on \(L^2(\mathbb{R}^3)\) and establish Breit--Wigner-type asymptotics for the spectral density of $H_\alpha$ along the resonance  $\lambda(\alpha)$ near a simple embedded eigenvalue $\lambda_0$ of $H_a$ as $\alpha\to a$. We also obtain similar asymptotic behaviour for the scattering cross-section and the average time delay.
\end{abstract}

\section{Introduction}
The resonance phenomenon associated with the disappearance of an embedded eigenvalue into the absolutely continuous spectrum under perturbation has been studied by various authors; see, for example, \cite{howland, Orth, Astaburuaga2024}.

In~\cite{LS1}, a detailed analysis was carried out for rank-one perturbations of the Laplacian on \(L^2(\mathbb{R}^3)\). It was shown that the disappearance of a simple embedded eigenvalue gives rise to Breit--Wigner-type asymptotic behaviour of the spectral density, scattering cross-section, and time delay near the resonance. In the present paper, we consider the finite-rank multi-parameter perturbations $H_\alpha$ of the Laplacian in $\RR^3$ such that at $\alpha=a$, $\lambda_0$ is a simple embedded eigenvalue of $H_a$ which disappears into the absolutely continuous spectrum after perturbation. We then show that, after a suitable translation and scaling of the energy variable, the spectral density of $H_\alpha$ exhibits Breit--Wigner-type asymptotics near the embedded eigenvalue $\lambda_0$ of $H_a$ as $\alpha\to a$. Similar asymptotic behaviour is also obtained for the scattering cross-section and the average time delay.

\section{Preliminaries}\label{2}
In this section we present the spectral representation of the Laplacian 
\(\Delta\) in \(\mathbb{R}^{3}\) and the boundary behaviour of its resolvent.

Let \(H_0\) denote the self-adjoint extension of \(-\Delta\) on the Hilbert space \(L^2(\mathbb{R}^3)\), with domain $H^2(\RR^3),$ the Sobolev space.
The operator \(H_0\) is purely absolutely continuous with spectrum \(\sigma(H_0) = [0,\infty)\). The spectral representation of \(H_0\) is given by the unitary operator \(U \colon L^2(\mathbb{R}^3) \to L^2([0,\infty); L^2(S^2))\) defined by
\begin{equation}\label{iso}
    (Uv)_\lambda(\omega)
    = 2^{-1/2} \lambda^{1/4} \widehat{v}(\sqrt{\lambda} \, \omega)
    \quad \text{for almost every } (\lambda,\omega) \in [0,\infty)\times S^2
\end{equation}where $\hat v$ denotes the Fourier transform of $v$ in $L^2$-sense.

For \( v \in L^2(\mathbb{R}^3) \), we simply write \( v_\lambda \) for \( (Uv)_\lambda \). For \( v, w \in L^2(\mathbb{R}^3) \), define the function 
\( \eta_{v,w} \colon \mathbb{R} \to \mathbb{C} \) by
\[
\eta_{v,w}(\lambda)
=
\begin{cases}
\langle v_\lambda, w_\lambda \rangle_{L^2(S^2)}, & \text{if } \lambda > 0, \\[0.3em]
0, & \text{if } \lambda \leq 0.
\end{cases}
\]
Let \( R_0(z) \) denote the resolvent operator of \( H_0 \), 
defined for \(z \in \mathbb{C} \setminus [0,\infty) \). Then for \( \Im z \neq 0 \), we have
\[
\langle R_0(z) v, w \rangle
= \int_0^\infty \langle (R_0(z) v)_\lambda, w_\lambda \rangle \, d\lambda
= \int_{\mathbb{R}} \frac{\eta_{v,w}(\lambda)}{\lambda - z} \, d\lambda.
\]

To study the resolvent limits of \( H_0 \), we restrict our attention to the following dense subspace of \( L^2(\mathbb{R}^3) \). Define
\[
\mathcal{E}
:=
\left\{
v \in L^2(\mathbb{R}^3)
\;\middle|\;
\begin{array}{l}
\text{the map } \lambda \mapsto v_\lambda \text{ is smooth with respect to } L^2(S^2)\text{-norm,} \\[0.2em]
\text{and } \|v_\lambda\|_{L^2(S^2)} \text{ decays rapidly as } \lambda\to 0^+ \text{ and } \lambda\to\infty
\end{array}
\right\}.
\]

For $v\in\mathcal E$, we write $v_s^{(k)}$ for the \(k\)-th derivative of the map \(\lambda\mapsto v_\lambda\) in the norm topology at \(\lambda=s\). 
Observe further that, for \( v, w \in \mathcal{E} \), the function \( \eta_{v,w} \) belongs to the Schwartz class \( \mathcal{S}(\mathbb{R}) \), since the map 
\( \lambda \mapsto \langle v_\lambda, w_\lambda \rangle \) is smooth with all derivatives decaying rapidly as \( \lambda \to 0^+ \) and as \( \lambda \to \infty \), and \( \eta_{v,w} \) vanishes identically on \( (-\infty,0] \). Consequently, by the Plemelj-Privalov Theorem (see \cite[Proposition 2.4]{LS1}), for \( v, w \in \mathcal{E} \),
\begin{equation}\label{eqn-laplacianresolventlimit}
\langle R_0(\lambda \pm \iota 0) v, w \rangle 
:= \lim_{\epsilon \to 0^+} \langle R_0(\lambda \pm \iota \epsilon) v, w \rangle
= \gamma(\eta_{v,w}, \lambda) \pm \iota \pi \, \eta_{v,w}(\lambda)
\end{equation}
where $\gamma(f,\lambda)$ denotes the Cauchy principal value of \(f\) at \(\lambda \in \mathbb{R}\); see \cite[Proposition 2.2]{LS1} for its properties.

For \( v \in \mathcal{E} \) with \( v_a = 0 \) for some \( a> 0 \), we define the function \( \Phi_{v,a} \) by
\begin{equation}\label{modified}
(\Phi_{v,a})_\lambda =
\begin{cases}
\dfrac{v_\lambda}{\lambda - a}, & \text{if } \lambda \neq a, \\[6pt]
v'_a, & \text{if } \lambda = a.
\end{cases}
\end{equation}
 
\noindent It can be seen that $\Phi_{v,a} \in \mathcal E$ and if $w \in \mathcal{E}$ such that
$w_{a}=0$, then \begin{equation}\label{ch5-lem1(c)}
\left.\frac{\partial}{\partial \lambda}
\gamma(\eta_{v,w},\lambda)\right|_{\lambda=a}
=
\langle \Phi_{v,a},\, \Phi_{w,a} \rangle.
\end{equation}
For more details, see \cite[Lemma 3.1]{LS2}.

\section{multi-parameter perturbations of the Laplacian }\label{3}In this section, we investigate the spectral properties of multi-parameter rank-\(n\) 
self-adjoint perturbations of the operator \(H_0\). Let
\[
\mathbb{R}_+^n
=
\{x=(x_1,\ldots,x_n)\in\mathbb{R}^n : x_j>0 \text{ for } j=1,\ldots,n\}
\]
denote the positive orthant of \(\mathbb{R}^n\). For \(\alpha = (\alpha_1, \dots, \alpha_n) \in \mathbb{R}^n_+\), consider the self-adjoint operator
\[
V_\alpha := \sum_{j=1}^n \alpha_j \langle \cdot, u_j \rangle u_j,
\]
where \(\{u_1, \dots, u_n\} \subset \mathcal{E}\) is an orthogonal family in \(L^2(\mathbb{R}^3)\). For notational simplicity, we write \((u_j)_\lambda\) as \(u_{j,\lambda}\) and \(\eta_{u_j,u_k}\) as \(\eta_{jk}\) for \(j,k \in \{1,\cdots,n\}\). Then the perturbed operator
\begin{equation}\label{multi}
H_{{\alpha}} := H_0 + V_{{\alpha}}
\qquad \text{on } L^2(\mathbb{R}^3)
\end{equation}
is self-adjoint with domain \(\mathcal{D}(H_{{\alpha}}) = H^2(\RR^3)\). Since the finite-rank operator $V_{{\alpha}}$ is positive for any ${{\alpha}}\in\RR^n_+$, we have
\[
\sigma(H_{{\alpha}}) = \sigma_{\mathrm{ess}}(H_{{\alpha}}) = [0,\infty).
\]
It is straightforward to check that \(0\) cannot be an eigenvalue of \(H_{ {\alpha}}\). For \(\Im z \neq 0\), we denote by \(R_{ {\alpha}}(z)\) the resolvent operator of \(H_{ {\alpha}}\), and by \(E_{ {\alpha}}\) its associated spectral measure.

To carry out a finer spectral analysis of \(H_{{\alpha}}\), we first derive a
relation between the resolvents of \(H_{{\alpha}}\) and \(H_0\). 
Define the linear map
 $\tau_{{\alpha}} : L^{2}(\mathbb{R}^{3}) \to \mathbb{C}^n$ by
\begin{equation}\label{tau}
\tau_{ {\alpha}} v :=\left(\sqrt{\alpha_1}\langle v,u_{1}\rangle,\cdots,\sqrt{\alpha_n}\langle v,u_n\rangle\right)^\top,\qquad v\in L^2(\mathbb{R}^3).
\end{equation}
Then $\tau_{{\alpha}}^{*}\tau_{{\alpha}} = V_{{\alpha}}.$ For $\Im z \neq 0$, the resolvent identity yields
\begin{equation}\label{comp101}
R_{{ {\alpha}}}(z) - R_{0}(z)
=
- R_{0}(z) V_{{\alpha}} R_{{{\alpha}}}(z)=- R_{0}(z) \tau_{{\alpha}}^{*}\tau_{{\alpha}} R_{{{\alpha}}}(z).
\end{equation}Applying $\tau_{{\alpha}}$ to \eqref{comp101}, we obtain $B({{\alpha}},z)\,\tau_{{\alpha}} R_{{ {\alpha}}}(z) 
=
\tau_{{\alpha}} R_{0}(z),
$
where
\begin{equation}\label{ch5-1eqn1}
B({ {\alpha}},z) := I + \tau_{ {\alpha}} R_{0}(z)\tau_{ {\alpha}}^{*},\qquad \Im z\neq 0.
\end{equation}
Since
$B(\alpha,z)\bigl(I - \,\tau_\alpha R_{\alpha}(z)\tau_\alpha^{*}\bigr)
=
\bigl(I - \tau_\alpha R_{\alpha}(z)\tau_\alpha^{*}\bigr) B(\alpha,z)
=
I,$
the operator $B(\alpha,z)$ is invertible for $\Im z \neq 0$. Consequently, we obtain
\begin{equation}\label{reducedresolventrelation}
\tau_\alpha R_{\alpha}(z)
=
B(\alpha,z)^{-1}\,\tau_\alpha R_{0}(z).
\end{equation}

Let \( b_{jk}(\alpha,z) \) denote the \((j,k)\)-th matrix coefficient of \( B(\alpha,z) \) with respect to the standard orthonormal basis \( \{e_1,\dots,e_n\} \) of \( \mathbb{C}^n \), and let \( \tilde{b}_{jk}(\alpha,z) \) denote the \((j,k)\)-th entry of the corresponding cofactor matrix. The next lemma relates the matrix elements of the resolvents of \( H_\alpha \) and \( H_0 \).
\begin{lemma}\label{newlemma}
    For $v,w\in L^2(\RR^3)$ and $\Im z\neq0$, we have 
    \begin{equation}\label{R1}
        \langle  R_\alpha(z)v,w\rangle =\langle R_0(z)v,w\rangle-\frac{1}{\det\,B(\alpha,z)}\sum_{j,k=1}^n\sqrt{\alpha_j}\sqrt{\alpha_k}\tilde{b}_{jk}(\alpha,z)\langle R_0(z)v,u_j\rangle \langle R_0(z)u_k,w\rangle.
    \end{equation}
\end{lemma}
\begin{proof}
    By \eqref{comp101} and \eqref{reducedresolventrelation},   \begin{align*}\langle R_\alpha(z)v,w\rangle-\langle R_0(z)v,w\rangle&=-\langle R_0(z)\tau_\alpha^*\tau_\alpha R_\alpha(z)v,w\rangle=- \langle \tau_\alpha R_\alpha(z)v,\tau_\alpha R_0(\overline{z})w\rangle\\&=-\langle B(\alpha,z)^{-1} \tau_\alpha R_0(z)v,\tau_\alpha R_0(\overline{z})w\rangle.\end{align*}
By the definition of the linear map $\tau_\alpha$, \eqref{R1} follows.
\end{proof}
\noindent Now by \eqref{ch5-1eqn1} and the fact that $\tau_\alpha^{*}e_j =\sqrt{\alpha_j} u_j$, \begin{align*}
b_{jk}(\alpha,z)&=\langle B(\alpha,z)e_k,\,e_j\rangle
=
\langle e_k,\,e_j\rangle
+
\langle \tau_\alpha R_{0}(z)\tau_\alpha^{*} e_k,\, e_j\rangle \\
&=
\delta_{kj}
+
\langle R_{0}(z)\tau_\alpha^{*} e_k,\, \tau_\alpha^{*} e_j\rangle =
\delta_{jk}
+
\sqrt{\alpha_j}\sqrt{\alpha_k}\langle R_{0}(z)u_k,\,u_j\rangle.
\end{align*}
\noindent Since $u_j \in \mathcal{E}$ for each $j$,
it follows from~\eqref{eqn-laplacianresolventlimit} that for any $\lambda\in\RR$
\begin{equation}\label{uiresolventlimit}
\lim_{\epsilon \to 0^{+}}
b_{jk}(\alpha,\lambda \pm \iota\epsilon)
=
\left(\delta_{jk}
+
\sqrt{\alpha_j}\sqrt{\alpha_k}\,\gamma(\eta_{kj},\lambda)\right)
\pm
\iota\pi\sqrt{\alpha_j}\sqrt{\alpha_k}\,\eta_{kj}(\lambda).
\end{equation} 
\noindent{\bf{Definition of $B(\alpha,\lambda\pm\iota0)$:}} For $\alpha\in\RR^n_+$ and $\lambda \in \mathbb{R}$, define the operator {$B(\alpha,\lambda\pm\iota0)$} whose matrix coefficients are given by $b_{jk}(\alpha,\lambda \pm \iota 0):=\lim_{\epsilon\to0^+}b_{jk}(\alpha,\lambda \pm \iota \epsilon)$.\\

\noindent By definition, 
$B(\alpha,\lambda+\iota0)
=
C(\alpha,\lambda) + \iota\, D(\alpha,\lambda),$
where $C(\alpha,\lambda)$ and $D(\alpha,\lambda)$ are self-adjoint operators with matrix coefficients
\begin{equation}
c_{jk}(\alpha,\lambda)
=
\delta_{jk}
+
\sqrt{\alpha_j}\sqrt{\alpha_k}\,\gamma(\eta_{kj},\lambda)\quad\text{and}\quad
d_{jk}(\alpha,\lambda)
=
\sqrt{\alpha_j}\sqrt{\alpha_k}\pi\,\eta_{kj}(\lambda).
 \end{equation} By \eqref{ch5-1eqn1}, we have $B(\alpha,z)^*=B(\alpha,\overline{z})$ for $\Im z\neq0$ and hence \[B(\alpha,\lambda-\iota0)=B(\alpha,\lambda+\iota0)^*=C(\alpha,\lambda)-\iota D(\alpha,\lambda).\]
 We denote by $\tilde b_{jk}(\alpha,\lambda\pm\iota0)$ and $\tilde c_{jk}(\alpha,\lambda)$, the $(j,k)$-th entries of the cofactor matrices
associated with $B(\alpha,\lambda\pm\iota0)$ and $C(\alpha,\lambda)$ respectively.
Define
\[
F(\alpha,\lambda)
:=
\det B(\alpha,\lambda+\iota0)
=
F_{1}(\alpha,\lambda) + \iota\, F_{2}(\alpha,\lambda),
\]
where $F_{1}(\alpha,\lambda) := \Re \det B(\alpha,\lambda+\iota 0)$ and
$F_{2}(\alpha,\lambda) := \Im \det B(\alpha,\lambda+\iota 0)$.

Note that for \(v,w\in \mathcal E\), by \eqref{R1} and \eqref{eqn-laplacianresolventlimit}, the boundary value
\[
\langle R_\alpha(\lambda\pm\iota0)v,w\rangle
:=
\lim_{\epsilon\to0^+}\langle R_\alpha(\lambda\pm\iota\epsilon)v,w\rangle
\]
exists if and only if \(B(\alpha,\lambda\pm\iota0)\) is invertible.

%


 On the other hand if the \(\operatorname{Rank} B(a,\lambda_0 + i0) = n - k\) for some fixed \(a =(a_1,a_2,\cdots,a_n)\in \mathbb{R}^n_+\), \(\lambda_0 > 0\), and \(0 < k \leq n\) and \(u_{j,\lambda_0} = 0\) for all \(j\), then \(D(a,\lambda_0) = 0\), which implies \(\operatorname{Rank} C(a,\lambda_0) = n - k\). Since \(C(a,\lambda_0)\) is self-adjoint, the spectral theorem yields a unitary matrix \(W = (w_{ij})_{i,j=1}^n\) such that
\begin{equation}\label{W}
W^* C(a,\lambda_0) W = \operatorname{diag}(0, \dots, 0, \mu_{k+1}, \mu_{k+2}, \dots, \mu_n),
\end{equation}
where \(\mu_{k+1}, \dots, \mu_n\) are the nonzero eigenvalues of \(C(a,\lambda_0)\), counted according to multiplicity. Let \(\mu := \mu_{k+1} \cdots \mu_n\) be the product of the nonzero eigenvalues of \(C(a,\lambda_0)\), and let \(\phi_j := \Phi_{u_j, \lambda_0}\) for $1\leq j\leq n$. 

We have the following lemma which may be of independent interest.  

\begin{lemma}\label{evlemma}
Assume that \(\operatorname{Rank} B(a,\lambda_0 + i0) =n - k\) for some fixed \(a \in \mathbb{R}^n_+\), \(\lambda_0 > 0\), and \(0 < k \leq n\), and \(u_{j,\lambda_0} = 0\) for all \(j\). Then \(\lambda_0\) is an eigenvalue of \(H_a\) of multiplicity \(k\), and the corresponding eigenspace is spanned by the vectors
\[
v_j := \sum_{l=1}^n \sqrt{a_l}\,w_{lj}\,\phi_l, \qquad j = 1, \dots, k.
\] 
Furthermore $v_j\in\mathcal E$ for all $j.$
\end{lemma}
\begin{proof}
By the preceding discussion and \eqref{W}, we have
\[
(w_{1j}, w_{2j}, \dots, w_{nj})^\top \in \ker C(a,\lambda_0) = \ker B(a,\lambda_0 + i0), \qquad j = 1, \dots, k.
\]
Then by an argument similar to that of Birman-Schwinger-type principle \cite[Theorem~4.2]{LS2}, the vectors $
v_j := \Phi_{\tau_a^*(w_{1j}, \dots, w_{nj})^\top, \lambda_0}$, $j = 1, \dots, k$
form an eigenbasis of \(H_a\) corresponding to the eigenvalue \(\lambda_0\). We observe that $v_j = \Phi_{\tau_a^*(w_{1j}, \dots, w_{nj})^\top, \lambda_0}
= \sum_{l=1}^n \sqrt{a_l} \, w_{lj} \, \Phi_{u_l, \lambda_0}$ and the lemma follows.
\end{proof}
We next derive the asymptotic behaviour of \(F_1(\alpha,\lambda)\) and \(F_2(\alpha,\lambda)\) as \(\lambda \to \lambda_0\); this result will be used later.

Let \(r\) be the largest integer such that \(u_{j,\lambda}\) vanishes to order at least \(r\) at \(\lambda = \lambda_0\) for every \(j = 1, \dots, n\). Then $d_{jk}(\alpha,\lambda)
= \sqrt{\alpha_j \alpha_k} \, \pi \, \eta_{kj}(\lambda)
= O\!\bigl(|\lambda - \lambda_0|^{2r}\bigr)$ as $\lambda \to \lambda_0$,
uniformly for \(\alpha\) in compact subsets of \(\mathbb{R}^n_+\). Thus by the multilinearity of the determinant with respect to the columns, we have:
\begin{equation} \label{ch5-lemF1F2expression(c)}
\tilde{b}_{jk}(\alpha,\lambda \pm i0) = \tilde{c}_{jk}(\alpha,\lambda) + O(|\lambda - \lambda_0|^{2r})
\qquad \text{as } \lambda \to \lambda_0,
\end{equation}
\begin{equation} \label{ch5-lemF1F2expression(a)}
F_1(\alpha,\lambda) = \det C(\alpha,\lambda) + O(|\lambda - \lambda_0|^{4r})
\qquad \text{as } \lambda \to \lambda_0,
\end{equation}
\begin{equation} \label{ch5-lemF1F2expression(b)}
F_2(\alpha,\lambda) = \sum_{j,k=1}^n d_{jk}(\alpha,\lambda) \tilde{c}_{jk}(\alpha,\lambda) + O(|\lambda - \lambda_0|^{4r})
\qquad \text{as } \lambda \to \lambda_0,
\end{equation}
uniformly for \(\alpha\) in compact subsets of \(\mathbb{R}^n_+\).

\section{Resonance under multi-parameter perturbation of a Simple embedded eigenvalue}\label{4}
We fix a model $H_{\alpha}=H_{0}+ V_\alpha$ on $L^{2}(\mathbb{R}^{3})$
such that $\lambda_{0}$ is a simple embedded eigenvalue of
$H_{{a}}$ which disappears into the absolutely continuous spectrum
under perturbation. We study the resonance phenomenon by analyzing the spectral density of $H_\alpha$ near $\lambda_0$ as $\alpha\to a$. The scattering cross-section and time delay for the pair $(H_0,H_\alpha)$ near $\lambda_0$ as $\alpha\to a$ is also studied.   

Let \(a\in \mathbb{R}^n_+\) and \(\lambda_0>0\) be fixed, and suppose that $
\operatorname{Rank} B(a,\lambda_0+\iota0)=n-1
$ and that $u_{j,\lambda_0}=0$ for any $j.$ Then by Lemma \ref{evlemma}, it follows that \(\lambda_0\) is a simple eigenvalue of \(H_a\) with correspoding eigenvector $\phi=\sum_{j=1}^n \sqrt{a_j}\,w_{j1}\,\phi_j$.

For \(\alpha\in\mathbb{R}_+^n\), define
$
\mathfrak{C}(\alpha,\lambda)
=
(\mathfrak c_{jk}(\alpha,\lambda))_{j,k=1}^n
:=
W^{*}C(\alpha,\lambda)W.$
Consequently,
\begin{equation}\label{H1}
\mathfrak c_{jk}(\alpha,\lambda)
=
\sum_{l,m=1}^n
\overline{w_{lj}}\,c_{lm}(\alpha,\lambda)\,w_{mk}.
\end{equation} 
Furthermore, $\operatorname{adj}\left(\mathfrak{C}(\alpha,\lambda)\right)=W^{*}\operatorname{adj}\left(C(\alpha,\lambda)\right)W$ where \(\operatorname{adj}\) denotes the adjugate matrix. This implies \begin{equation}\label{H3}{ \tilde c_{jk}}({\alpha},\lambda)=\sum_{l,m=1}^n \overline{w_{jl}}\tilde{\mathfrak c}_{lm}(\alpha,\lambda)w_{km} \end{equation}where \(\tilde{\mathfrak c}_{lm}(\alpha,\lambda)\) denotes the \((l,m)\)-cofactor of \(\mathfrak C(\alpha,\lambda)\). Using the fact that
\[
\tilde{\mathfrak c}_{lm}(a,\lambda_0)
=
\begin{cases}
\mu, & l=m=1,\\
0, & \text{otherwise},
\end{cases}
\]
we obtain
\begin{equation}\label{H2}
\tilde c_{jk}(a,\lambda_0)
=
\mu\,\overline{w_{j1}}\,w_{k1}.
\end{equation}
 
\noindent In the next lemma, we analyze the zeros of the function $F_1(\alpha,\cdot)$ for $\alpha$ near ${a}$.

\begin{lemma}\label{ch5-lambdaderivative}
Suppose that \(u_{j,\lambda_0}=0\) for all \(j\), and that
$\operatorname{Rank} B({a},\lambda_{0}+\iota0)=n-1$. Then there exist an open ball \(I\subset \mathbb{R}_+^n\) centered at \(a\), an open interval \(J\subset(0,\infty)\) containing \(\lambda_0\), and a unique smooth function $\lambda:I\to J$ such that $
\lambda(a)=\lambda_0 \text{ and } 
F_1(\alpha,\lambda(\alpha))=0 \text{ for all }\alpha\in I.
$

\end{lemma}
\begin{proof}
Since \(C(\alpha,\lambda)\) and \(\mathfrak C(\alpha,\lambda)\) are unitarily equivalent, by \eqref{H1}, \[
\frac{\partial F_1}{\partial\lambda}(a,\lambda_0)
=
\left.\frac{\partial}{\partial\lambda}\det \mathfrak C(\alpha,\lambda)\right|_{(a,\lambda_0)}=
\mu\,\frac{\partial \mathfrak c_{11}}{\partial\lambda}(a,\lambda_0)
=\mu\sum_{l,m=1}^n
\overline{w_{l1}}
\frac{\partial c_{lm}}{\partial\lambda}(a,\lambda_0)
w_{m1}.
\]
Thus by \eqref{ch5-lem1(c)}, we obtain 
          \begin{equation} \label{pathderivative}  \frac{\partial F_1}{\partial\lambda}({a},\lambda_0)=\mu\sum_{l,m=1}^n \sqrt{a_m}\sqrt{a_l}\langle \phi_m,\phi_l\rangle \overline{w_{l1}}w_{m1}=\mu\|\phi\|^2\neq 0.  \end{equation}   Since \(F_1(a,\lambda_0)=0\), the proof follows by the implicit function theorem.  
\end{proof}

\subsection{The model} 

Along with hypotheses of Lemma~\ref{ch5-lambdaderivative}, if we assume that
$
F_{2}(\alpha,\lambda(\alpha)) \neq 0
\text{ for all } \alpha \in I\setminus\{a\},$
then by an argument similar to that of \cite[Theorem 4.7 (a)]{LS2}, it follows that \(F(\alpha,\lambda)\neq 0\) for all \(\alpha\in I\setminus\{a\}\) and \(\lambda\in J\). Furthermore, for each \(\alpha \in I\setminus\{a\}\), the spectrum of \(H_{\alpha}\) in \(J\) is purely absolutely continuous, while the spectrum of \(H_a\) is purely absolutely continuous in \(J\setminus\{\lambda_0\}\). 

This motivates us to fix a model $H_{\alpha}=H_{0}+ V_\alpha$ on $L^{2}(\mathbb{R}^{3})$
such that $\lambda_{0}$ is a simple embedded eigenvalue of
$H_{{a}}$ which disappears into the absolutely continuous spectrum
under perturbation. 
 
\begin{definition}[The model]\label{model}
Let \(a\in \mathbb{R}_+^n\) and \(\lambda_0>0\) be fixed. For \(\alpha\in\mathbb{R}_+^n\), consider the operator
\begin{equation}\label{model22}
H_\alpha=H_0+V_\alpha
\quad \text{on } L^2(\mathbb{R}^3),
\end{equation}
where
$V_\alpha=\sum_{j=1}^n \alpha_j \langle \cdot,u_j\rangle u_j$,
and assume the following:
\begin{enumerate}[label=(\alph*)]
    \item The family \(\{u_j\}_{j=1}^n\subset \mathcal E\) is orthogonal.
    
    \item \(u_{j,\lambda_0}=0\) for all \(j=1,\dots,n\) (equivalently, \(D(a,\lambda_0)=0\)).
    
    \item \(\operatorname{Rank} B(a,\lambda_0+\iota0)=n-1\).
    
    \item \(F_2(\alpha,\lambda(\alpha))\neq 0\) for all \(\alpha\in I\setminus\{a\}\), where \(I\), \(J\), and \(\lambda:I\to J\) are as in Lemma~$\ref{ch5-lambdaderivative}$.
\end{enumerate}
\end{definition}

In the subsequent subsections, all the results will be with respect to this model.

\subsection{Spectral density and its asymptotic behaviour} 

The operator \(H_\alpha\) defined in Definition \ref{model} is purely absolutely continuous in the interval \(J\) for all
\(\alpha\in I\setminus\{a\}\). For \(v,w\in L^2(\mathbb{R}^3)\), let \(\rho_\alpha^{v,w}\) denote the spectral density on \(J\) associated with the spectral measure \(\langle E_\alpha(d\lambda)v,w\rangle\). Then for \(v,w\in\mathcal E\), the density \(\rho_\alpha^{v,w}\) is given by

\begin{equation}\label{R9}
\rho_\alpha^{v,w}(\lambda)
=
\frac{1}{2\pi\iota}
\left(
\langle R_\alpha(\lambda+\iota0)v,w\rangle
-
\langle R_\alpha(\lambda-\iota0)v,w\rangle
\right),
\qquad \lambda\in J.
\end{equation}
 
Now to exhibit the Breit--Wigner-type behaviour of \(\rho_\alpha^{v,w}\)  near the embedded eigenvalue \(\lambda_0\) as \(\alpha\to a\), define the scaling function \(\kappa:I\to\mathbb R\) by
\begin{equation}\label{ch5-eqn-kappa}
\kappa(\alpha)
:=
\frac{F_2(\alpha,\lambda(\alpha))}{\mu\|\phi\|^2}.
\end{equation}
Then \(\kappa(a) = 0\) and \(\kappa(\alpha) > 0\) for all \(\alpha\) sufficiently close to \(a\) with \(\alpha \neq a\). Furthermore
\begin{equation}\label{ch5-kappaorder}
\kappa(\alpha)=O(|\alpha-a|^{2r})\qquad\text{as }\alpha\to a
\end{equation}where \(r\) is the largest integer such that \(u_{j,\lambda}\) vanishes to order at least \(r\) at \(\lambda = \lambda_0\) for any $j.$
 For any fixed \(h\in\mathbb R\), define the scaled and translated spectral parameter
\[ 
\lambda_h(\alpha)
=
\lambda(\alpha)+h\,\kappa(\alpha),
\qquad \alpha\in I.
\]Next, observe that since \(F_1(a,\lambda_0) = 0\), for any fixed \(h \in \mathbb{R}\), by \eqref{pathderivative}, we have
\begin{equation}\label{eqn-F1limit1}
\lim_{\alpha \to a} \frac{F_1(\alpha, \lambda_h(\alpha))}{\kappa(\alpha)}
=h\frac{\partial F_1}{\partial\lambda}(a,\lambda_0)= h \mu \|\phi\|^2.\end{equation} Similarly using \eqref{ch5-eqn-kappa}, we have\begin{equation}\label{eqn-F2limit1}
       \lim\limits_{\alpha\to{a}}\dfrac{F_2(\alpha,\lambda_h(\alpha))}{\kappa(\alpha)}=\mu\|\phi\|^2.
    \end{equation} 
The following theorem describes the asymptotic behaviour of the spectral density of $H_\alpha$ in terms of the Cauchy distribution as $\alpha\to a$.

\begin{theorem}\label{cmpthm1}
For any fixed \(h \in \mathbb{R}\) and \(v, w \in \mathcal{E}\), we have
\begin{equation}\label{uidensityasy1}
\lim_{\alpha\to{a}} \kappa(\alpha) \rho^{v,w}_\alpha(\lambda_h(\alpha))
= \frac{\langle P_{\lambda_0} v, w \rangle}{\pi}  \frac{1}{h^2 + 1},
\end{equation}
where \(P_{\lambda_0}\) denotes the eigenprojection corresponding to the eigenvalue \(\lambda_0\) of \(H_{{a}}\).

\end{theorem}

\begin{proof}
First note that, by \eqref{R1}, for any \(\alpha \in I \setminus \{{a}\}\),
\begin{equation}\label{R4}
\begin{split}
&\kappa(\alpha) \langle R_\alpha(\lambda_h(\alpha) + \iota 0) v, w \rangle\\
&= \kappa(\alpha) \langle R_0(\lambda_h(\alpha) + \iota 0) v, w \rangle \\&\quad-\frac{\kappa(\alpha)}{F(\alpha, \lambda_h(\alpha))}
\sum_{j,k=1}^n\Big\{ \sqrt{\alpha_j}\sqrt{\alpha_k}{\tilde{b}_{jk}(\alpha, \lambda_h(\alpha) + \iota 0)}\langle R_0(\lambda_h(\alpha) + \iota 0) v, u_j \rangle \\&\hspace{8cm}\times\langle R_0(\lambda_h(\alpha) + \iota 0) u_k, w \rangle\Big\}.
\end{split}
\end{equation}
By \eqref{ch5-lemF1F2expression(c)} and \eqref{H2}, we have
\begin{equation}\label{cofactorlimit}
\lim_{\alpha\to{a}}  {\tilde{b}_{jk}(\alpha, \lambda_h(\alpha) + \iota 0)} 
=  {\mu\overline{w_{j1}} w_{k1}} .
\end{equation}

Taking the limit as \(\alpha \to {a}\) in \eqref{R4} and using \eqref{eqn-laplacianresolventlimit}, \eqref{eqn-F1limit1} and \eqref{eqn-F2limit1}, together with \eqref{cofactorlimit}, we obtain
\begin{equation}\label{R7}
\begin{split}
\lim_{\alpha\to{a}} \kappa(\alpha) \langle R_\alpha(\lambda_h(\alpha) + \iota 0) v, w \rangle
&= -\frac{1}{\mu\|\phi\|^2(h + \iota)} \sum_{j,k=1}^n\mu\sqrt{a_j}\sqrt{a_k}\overline{w_{j1}} w_{k1} \langle v, \phi_j \rangle \langle \phi_k, w \rangle \\
&= -\frac{\langle v, \phi \rangle \langle \phi, w \rangle}{\|\phi\|^2(h + \iota)}
= -\frac{\langle P_{\lambda_0} v, w \rangle}{h + \iota}.
\end{split}
\end{equation}
By an analogous argument, we obtain
\begin{equation}\label{R8}
\lim_{\alpha\to{a}} \kappa(\alpha) \langle R_\alpha(\lambda_h(\alpha) - \iota 0) v, w \rangle
= -\frac{\langle P_{\lambda_0} v, w \rangle}{h - \iota}.
\end{equation}
The conclusion now follows from \eqref{R7}, \eqref{R8} and \eqref{R9}.
\end{proof}
\begin{remark}Spectral concentration, time-decay behaviour, and a lower bound on the sojourn time can be given as in 
\cite[Corollaries 4.6 and 4.8, and Theorem 5.4]{LS1}.\end{remark}

\subsection{Behaviour of scattering amplitude and scattering cross-section}Let $S^{(\alpha)}_\lambda$ denote the scattering matrix for the pair
$(H_0,H_\alpha)$ at ``energy'' $\lambda$. Let
$R^{(\alpha)}_\lambda$ and $\sigma_\alpha(\lambda)$ denote the
corresponding scattering amplitude operator and scattering cross-section
respectively. For the definitions and further
details one may consult \cite[Chapter 7]{KBSBook} . 

For $\alpha\in I\setminus\{{a}\},\,\lambda\in J,$ the scattering amplitude operator $R^{(\alpha)}_\lambda$ is given by:
\begin{equation}\label{laplacianscatteringmatrix}
R^{(\alpha)}_\lambda
=
-\frac{2\pi \iota}{F(\alpha,\lambda)}
\left(
\sum_{j,k=1}^n\sqrt{\alpha_j}\sqrt{\alpha_k}
\tilde{b}_{jk}(\alpha,\lambda+\iota0)
\langle \,\cdot\,,u_{j,\lambda}\rangle\, u_{k,\lambda}
\right).
\end{equation}

\noindent The above formula can be proved along the same lines as that of \cite[Proposition~8.22]{KBSBook} where the concern of the authors was only rank-one perturbation.

In the next theorem we describe the asymptotic behaviour of
$R^{(\alpha)}_\lambda$ and $\sigma_\alpha(\lambda)$ near the embedded
eigenvalue $\lambda_0$ of $H_a$ as $\alpha\to a$.  
\begin{theorem}\label{Rthm}
Assume that \(u'_{j,\lambda_0} \neq 0\) for all \(j = 1, \cdots, n\). Then, for any fixed \(h \in \mathbb{R}\),
\begin{equation}\label{eqn-1113}
\lim_{\alpha \to {a}} R^{(\alpha)}_{\lambda_h(\alpha)}
= -\frac{2\iota}{h + \iota} 
\Bigl\langle \cdot, \frac{\phi_{\lambda_0}}{\|\phi_{\lambda_0}\|} \Bigr\rangle 
\frac{\phi_{\lambda_0}}{\|\phi_{\lambda_0}\|}
\end{equation}
in the Hilbert--Schmidt norm. Consequently, for any fixed \(h \in \mathbb{R}\),
\begin{equation}\label{eqn-1114}
\lim_{\alpha \to {a}} \sigma_\alpha(\lambda_h(\alpha)) = \frac{4\pi}{\lambda_0} \cdot \frac{1}{h^2 + 1}.
\end{equation}
\end{theorem}

\begin{proof}(Sketch)
Using \eqref{laplacianscatteringmatrix}, write
\begin{equation}\label{initial}
R^{(\alpha)}_{\lambda_h(\alpha)} = -2\iota \, I(\alpha,h) 
\sum_{j,k=1}^n I_{jk}(\alpha,h) 
\Bigl\langle \cdot, \frac{u_{j,\lambda_h(\alpha)}}{\|u_{j,\lambda_h(\alpha)}\|} \Bigr\rangle 
\frac{u_{k,\lambda_h(\alpha)}}{\|u_{k,\lambda_h(\alpha)}\|},
\end{equation}
where
\begin{equation}\label{eqn-I}
I(\alpha,h) = \frac{\kappa(\alpha) }{F(\alpha,\lambda_h(\alpha))},
\quad
I_{jk}(\alpha,h) = \frac{\sqrt{\alpha_j}\sqrt{\alpha_k} \pi \tilde{b}_{jk}(\alpha,\lambda_h(\alpha) + \iota 0) 
\|u_{j,\lambda_h(\alpha)}\| \|u_{k,\lambda_h(\alpha)}\|}
{\kappa(\alpha)  }.
\end{equation}
A similar argument as in \cite[Theorem 7.4]{LS1} together with the definition of the vectors \(\phi_j\) yields, for all \(j,k\)
\begin{equation}\label{eqn-1}
\lim_{\alpha \to {a}} 
\Bigl\langle \cdot, \frac{u_{j,\lambda_h(\alpha)}}{\|u_{j,\lambda_h(\alpha)}\|} \Bigr\rangle 
\frac{u_{k,\lambda_h(\alpha)}}{\|u_{k,\lambda_h(\alpha)}\|}
= \Bigl\langle \cdot, \frac{u'_{j,\lambda_0}}{\|u'_{j,\lambda_0}\|} \Bigr\rangle 
\frac{u'_{k,\lambda_0}}{\|u'_{k,\lambda_0}\|}
= \Bigl\langle \cdot, \frac{\phi_{j,\lambda_0}}{\|\phi_{j,\lambda_0}\|} \Bigr\rangle 
\frac{\phi_{k,\lambda_0}}{\|\phi_{k,\lambda_0}\|}
\end{equation}
in the Hilbert--Schmidt norm.

To compute \(\lim\limits_{\alpha \to {a}} I_{jk}(\alpha,h)\), use \eqref{ch5-lemF1F2expression(c)}  together with the fact that $u_{l,\lambda_0}=u_{k,\lambda_0}=0$ and substitute the expression for $\kappa(\alpha)$ to obtain\begin{equation}\label{comp11}  I_{jk}(\alpha,h)
        = \frac{\mu\|\phi\|^2\sqrt{\alpha_j}\sqrt{\alpha_k}\pi\tilde{c}_{jk}(\alpha,\lambda_h(\alpha))
                \|u_{j,\lambda_h(\alpha)}\|\,
                \|u_{k,\lambda_h(\alpha)}\|}
                {F_2(\alpha,\lambda(\alpha))}+O(|\alpha-{a}|^2).\end{equation}
  Use \eqref{ch5-lemF1F2expression(b)} to write \begin{align*}\frac{F_2(\alpha,\lambda(\alpha))}{ \pi\,
                \|u_{j,\lambda_h(\alpha)}\|\,
                \|u_{k,\lambda_h(\alpha)}\|}&=\dfrac{\displaystyle\sum_{l,m=1}^n\sqrt{\alpha_l}\sqrt{\alpha_m}\tilde{c}_{lm}(\alpha,\lambda(\alpha))
                \langle u_{m,\lambda(\alpha)},u_{l,\lambda(\alpha)}\rangle}{ 
                \|u_{j,\lambda_h(\alpha)}\|\,
                \|u_{k,\lambda_h(\alpha)}\|}+O(|\alpha-{a}|^2).\end{align*} 
Since $\tilde{c}_{lm}(a,\lambda_0)=\mu \overline{w_{l1}}w_{m1}$ and \(
\lim\limits_{\alpha \to {a}} 
\frac{\langle u_{n,\lambda_h(\alpha)}, u_{m,\lambda_h(\alpha)} \rangle}
{\|u_{j,\lambda_h(\alpha)}\| \|u_{k,\lambda_h(\alpha)}\|}
= \frac{\langle \phi_{n,\lambda_0}, \phi_{m,\lambda_0} \rangle}
{\|\phi_{j,\lambda_0}\| \cdot \|\phi_{k,\lambda_0}\|}
\), we obtain
\begin{align*}
\lim_{\alpha \to {a}} 
\frac{F_2(\alpha,\lambda(\alpha))}{  \pi 
\|u_{j,\lambda_h(\alpha)}\| \|u_{k,\lambda_h(\alpha)}\|}
= \frac{\mu\|\phi_{\lambda_0}\|^2}{  \|\phi_{j,\lambda_0}\|\cdot\|\phi_{k,\lambda_0}\|}.
\end{align*}
Use \eqref{comp11} along with the above to get
\begin{equation}\label{eqn-3}
\lim_{\alpha \to {a}} I_{jk}(\alpha,h) 
= \frac{\mu\sqrt{a_j}\sqrt{a_k} \overline{w_{j1}}w_{k1} \|\phi_{j,\lambda_0}\|\cdot \|\phi_{k,\lambda_0}\|\cdot\|\phi\|^2}
{ \|\phi_{\lambda_0}\|^2}.
\end{equation}
As \(\phi = \sum_{j=1}^n\sqrt{a_j} w_{j1} \phi_j\), take the limit in \eqref{initial}, substitute the limits  for $F_1/\kappa$ and $F_2/\kappa$ (from \eqref{eqn-F1limit1}, \eqref{eqn-F2limit1}) and use \eqref{eqn-1},\eqref{eqn-3} to  get
\begin{align*}
\lim_{\alpha \to {a}} R^{(\alpha)}_{\lambda_h(\alpha)}
&= -\frac{2\iota}{(h + \iota)\mu\|\phi\|^2 \|\phi_{\lambda_0}\|^2} 
\sum_{j,k=1}^n \sqrt{a_j}\sqrt{a_k} \overline{w_{j1}} w_{k1}\mu\|\phi\|^2 
\Bigl\langle \cdot, \phi_{j,\lambda_0}\Bigr\rangle \phi_{k,\lambda_0}\\&
= -\frac{2\iota}{h + \iota} 
\left\langle \cdot, \frac{\phi_{\lambda_0}}{\|\phi_{\lambda_0}\|} \right\rangle 
\frac{\phi_{\lambda_0}}{\|\phi_{\lambda_0}\|}.
\end{align*}
This proves \eqref{eqn-1113} which immediately implies \eqref{eqn-1114}.
\end{proof}
 \subsection{Behaviour of time delay}
Let $\zeta_\alpha$ denote the average time delay associated with the pair
$(H_0,H_\alpha)$. By an analogous argument as in \cite[Subsection 8.2]{LS2}, for $\lambda\in J$ \[
 \zeta_\alpha(\lambda)=  -2\,\frac{
F_1(\alpha,\lambda)\,\frac{\partial F_2}{\partial\lambda}(\alpha,\lambda)
-
F_2(\alpha,\lambda)\,\frac{\partial F_1}{\partial\lambda}(\alpha,\lambda)
}{
|F(\alpha,\lambda)|^2}.\]
 The next theorem describes the asymptotic behaviour
of the average time delay near the embedded eigenvalue $\lambda_0$ of
$H_a$ as $\alpha\to a$. 
\begin{theorem}\label{31}
For each fixed $h\in\mathbb{R}$,
\[
\lim_{\alpha\to{a}}
\kappa(\alpha)\,\zeta_\alpha(\lambda_h(\alpha))
=
\frac{2}
     {h^2+1}.
\]
\end{theorem}
\begin{proof}
Substitute $\lambda=\lambda_h(\alpha)$ into the expression for $\zeta_\alpha(\lambda)$ above to obtain
\begin{equation}\label{equation4}\begin{split}
&\kappa(\alpha)\,\zeta_\alpha(\lambda_h(\alpha))\\
&=
-2
\frac{
\kappa(\alpha)\,F_1(\alpha,\lambda_h(\alpha))
\,\frac{\partial F_2}{\partial\lambda}(\alpha,\lambda_h(\alpha))
}{
|F(\alpha,\lambda_h(\alpha))|^2
}+2\frac{
\kappa(\alpha)\,F_2(\alpha,\lambda_h(\alpha))
\,\frac{\partial F_1}{\partial\lambda}(\alpha,\lambda_h(\alpha))
}{
|F(\alpha,\lambda_h(\alpha))|^2
}.\end{split}
\end{equation}
\noindent By the asymptotic $F_2(\alpha,\lambda_h(\alpha))=O(|\alpha-\alpha_0|^2)$ (\eqref{ch5-lemF1F2expression(b)}), we have
\[
\lim\limits_{\alpha\to{a}}
 {
\frac{\partial F_2}{\partial\lambda}(\alpha,\lambda_h(\alpha))
} 
=
0.\] 
Recall from \eqref{pathderivative} that, \[
\lim\limits_{\alpha\to{a}}
\frac{\partial F_1}{\partial\lambda}(\alpha,\lambda_h(\alpha))
=
\frac{\partial F_1}{\partial\lambda}({a},\lambda_0)=
\mu||\phi||^2.
\] Take the limit in \eqref{equation4} and substitute the limits  for $F_1/\kappa$ and $F_2/\kappa$ (from \eqref{eqn-F1limit1}, \eqref{eqn-F2limit1}) and use the above two limits to obtain the result.
\end{proof}
 \section*{Acknowledgements}
 The authors gratefully acknowledge Professor K.B. Sinha for some fruitful discussions during the early stages of this work.

The first author acknowledges the support received from the National Board for Higher Mathematics (Department of Atomic Energy, Government of India) Ph.D. Scholarship, Grant No. 0203/7/2019/RD-II/14855, with which this work was initiated. He also gratefully acknowledges the support received from IISER Mohali to carry out this research work.

\bibliographystyle{alpha}

\begin{thebibliography}{99}\bibitem{KBSBook} W. O. Amrein, J. M. Jauch, K. B. Sinha, Scattering theory in quantum mechanics. Physical principles and mathematical methods, W. A. Benjamin, Inc., Reading, Mass.-London-Amsterdam, 1977, pp. 691.
\bibitem{Astaburuaga2024}
M. A. Astaburuaga, V. H. Cortés, C. Fernández and R. Del Río, Resonances and stability of absolutely continuous spectrum for finite rank perturbations, \emph{Pure Appl. Funct. Anal.}, vol. 9, no. 4, pp. 899--914, 2024. \bibitem{LS1}
H.~Bansal, A.~Maharana, L.~Sahu and K.~B.~Sinha,
``Shape-resonance in spectral density, scattering cross-section, time delay, and bound on sojourn time,''
\emph{J.~Math. Anal. Appl.}, vol.~558, no.~1, Article~130373, 2026. \url{https://doi.org/10.1016/j.jmaa.2025.130373}
\bibitem{LS2}
H.~Bansal, A.~Maharana, L.~Sahu,
``Resonance near a doubly degenerate embedded eigenvalue,''
\emph{ Manuscript  submitted for publication.}
\url{https://doi.org/10.48550/arXiv.2603.08554}
\bibitem{howland} J. S. Howland, Perturbation of embedded eigenvalues by operators of finite rank, J. Math. Anal. Appl. 23, 1968, 575--584. 
 \bibitem{Orth}
A. Orth, Quantum mechanical resonance and limiting absorption: the many body problem, \emph{Comm. Math. Phys.}, vol. 126, no. 3, pp. 559--573, 1990.   

\end{thebibliography}

\end{document}